\documentclass[reqno]{amsart}     
\usepackage{pb-diagram,enumerate,amsmath,amssymb}     

\let\nc\newcommand 

\theoremstyle{plain}      
 
\newtheorem{thm}{Theorem}[section]     
\newtheorem{theorem}[thm]{Theorem}

\newtheorem{lemma}[thm]{Lemma}     
\newtheorem{prop}[thm]{Proposition}     
\newtheorem{proposition}[thm]{Proposition}

\theoremstyle{definition}      
     
\newtheorem{definition}[thm]{Definition}     

\newcommand{\definedas}{\mathrel{\raise.095ex\hbox{\rm :}\mkern-5.2mu=}}

\renewcommand{\epsilon}{\varepsilon}

\def\lamin{\lambda_{\rm min}^+}    
\def\al{{\alpha}}         
\def\be{{\beta}}         
\def\de{{\delta}}         
         
\def\om{{\omega}}         
\def\Om{{\Omega}}         
\def\la{{\lambda}}

\def\Si{{\Sigma}}

\def\phi{{\varphi}}

\let\theta\vartheta
\let\phi\varphi

\def\cR{{\mathcal R}}
 
\DeclareMathAlphabet{\doba}{U}{msb}{m}{n}         
\gdef\mC{\doba{C}}         
         
\gdef\mN{\doba{N}}

\gdef\mR{\doba{R}}         
\gdef\mS{\doba{S}}         
         
\gdef\mZ{\doba{Z}}

\def\grad{{\mathop{\rm grad}}} 
\def\scal{{\mathop{\rm scal}}}
\def\Vol{{\mathop{\rm Vol}}}

\def\Spin{{\mathop{\rm Spin}}}
\def\SO{{\mathop{\rm SO}}} 
\def\supp{{\mathop{\rm supp}}}
\def\ker{{\mathop{\rm ker\,}}}

\def\dim{{\mathop{\rm dim}}} 

\def\Id{\operatorname{Id}} 
\let\spec\Spec 
 
\def\End{{\mathop{\rm End}}}

\def\cRneq{\cR_{p,U,\gfl}^{\neq 0}} \def\cRinv{\cR_{U,\gfl}^{\rm inv}}

  \def\eref#1{{\rm (\ref{#1})}}

\let\<\langle \let\>\rangle  \let\wihat\widehat

\nc{\PsoM}{\mbox{{$P_{\mbox{\scriptsize SO}}(M)$}}}
\nc{\PsonM}{\mbox{{$P_{\mbox{\scriptsize SO(n)}}(M)$}}}
\nc{\PsoG}{\mbox{{$P_{\mbox{\scriptsize SO}}(G)$}}}
\nc{\PspinM}{\mbox{{$P_{\mbox{\scriptsize Spin}}(M)$}}}
\nc{\PspinG}{\mbox{{$P_{\mbox{\scriptsize Spin}}(G)$}}}
\nc{\Pspineps}{\mbox{{$P_{\mbox{\scriptsize Spin,}\epsilon}$}}}

\def\gfl{g_{\rm flat}}

\begin{document}


\title{Mass endomorphism, surgery and perturbations}

\author{Bernd Ammann}
\address{Bernd Ammann, Fakult\"at f\"ur Mathematik \\
   Universit\"at Regensburg \\
   93040 Regensburg \\
   Germany} 
\email{bernd.ammann@mathematik.uni-regensburg.de}

\author{Mattias Dahl}
\address{Mattias Dahl, Institutionen f\"or Matematik \\
   Kungliga Tekniska H\"ogskolan \\
   100 44 Stockholm \\
   Sweden} 
\email{dahl@math.kth.se}

\author{Andreas Hermann}
\address{Andreas Hermann, Fakult\"at f\"ur Mathematik \\
   Universit\"at Regensburg \\
   93040 Regensburg \\
   Germany} 
\email{andreas.hermann@mathematik.uni-regensburg.de}

\author{Emmanuel Humbert}
\address{Emmanuel Humbert, Institut \'Elie Cartan, BP 239 \\
   Universit\'e de Nancy\\
   54506 Vandoeuvre-l\`es-Nancy Cedex \\
   France} 
\email{humbert@iecn.u-nancy.fr}


\date{September 28, 2010}


\begin{abstract}
We prove that the mass endomorphism associated to the Dirac operator
on a Riemannian manifold is non-zero for generic Riemannian
metrics. The proof involves a study of the mass endomorphism under
surgery, its behavior near metrics with harmonic spinors, and
analytic perturbation arguments. 
\end{abstract}

\maketitle

\section{Introduction}

Let $(M,g)$ be a compact Riemannian spin manifold, we always assume
that a spin manifold comes equipped with a choice of orientation and
spin structure. Assume that the metric $g$ is flat in a neighborhood
of a point $p \in M$ and has no harmonic spinors. Then the Green's
function $G^g$ at $p$ for the Dirac operator $D^g$ exists. The
constant term in the expansion of $G^g$ at $p$ is an endomorphism of
$\Si_p M$ called the mass endomorphism. The terminology is motivated
by the analogy to the ADM mass being the constant term in the Green's
function of the Yamabe operator. The non-nullity of the mass
endomorphism has many interesting consequences. In particular,
combining the results presented here with inequalities in
\cite{ammann.humbert.morel:06} and \cite{hijazi:91}, one obtains a
solution of the Yamabe problem.

Finding examples for which the mass endomorphism does not vanish is
then a natural problem. In \cite{hermann:10}, it is proven that for a
generic metric on a manifold of dimension $3$, the mass endomorphism
does not vanish in a given point $p$. The aim of this paper is to
extend this result to all dimensions at least $3$, see Theorem
\ref{main}.

\section{Definitions and main result}

The goal of this section is to give a precise statement of the main
results. At first, the mass endomorphism is defined. Then, in
Subsection \ref{flat_aroundp_metrics}, we define suitable sets of
metrics to work with. Further, in Subsection \ref{alpha_genus}, we
explain some well known facts on the $\alpha$-genus. Finally, in
Subsection \ref{mainresult} we state Theorem \ref{main}, which is the
main result of this article.

\subsection{Mass endomorphism}
\label{submassendo}

In this section we will recall the mass endomorphism introduced in
\cite{ammann.humbert.morel:06}. Let $(M,g)$ be a compact spin manifold
of dimension $n \geq 2$ and $p \in M$. Assume that the metric $g$ is
flat in a neighborhood of $p$ and that the Dirac operator $D^g$ is
invertible. The Green's function $G^g(p, \cdot) = G^g(\cdot)$ of $D^g$
at $p$ is defined by
\begin{equation*}
  D^g G^g = \delta_p \Id_{\Sigma_p M},
\end{equation*}
where $\delta_p$ is the Dirac distribution at $p$ and $G^g$ is viewed
as a linear map which associates to each spinor in $\Sigma_p M$ a
smooth spinor field on $M \setminus \{p\}$. The distributional
equation satisfied by $G^g$ should be interpreted as
\begin{equation*}
  \int_M \< G^g(x) \psi_0  , D^g \phi(x) \> \, dv^g (x)
  =
  \< \psi_0, \phi(p) \>
\end{equation*}
for any $\psi_0 \in \Sigma_p M$ and any smooth spinor field $\phi$.
Let $\xi$ denote the flat metric on $\mR^n$, it then holds that
\begin{equation*}
  G^{\xi} \psi = - \frac{1}{\omega_{n-1} |x|^n} x \cdot \psi.  
\end{equation*}
at $p=0$, where $\omega_{n-1}$ is defined as the volume of $S^{n-1}$.
The following Proposition is proved in
\cite{ammann.humbert.morel:06}.

\begin{prop}\label{prop.before.mass}
  Let $(M,g)$ be a compact spin manifold of dimension $n \geq 2$.
  Assume that $g$ is flat on a neighborhood $U$ of a point $p \in M$.
  Then, for $\psi_0 \in \Sigma_p M$ we have
  \begin{equation*}
    G^g (x) \psi_0 
    =  
    - \frac{1}{\omega_{n-1}|x|^n} x \cdot \psi_0
    + v^g(x) \psi_0,
  \end{equation*}
  where the spinor field $v^g(x) \psi_0$ satisfies $D^g ( v^g(x)
  \psi_0 ) = 0$ in a neighborhood of $p$.
\end{prop}

This allows us to define the mass endomorphism.
\begin{definition}\label{def.mass}
  The {\it mass endomorphism} $\alpha^g: \Sigma_p M \to \Sigma_p M$
  for a point $p \in U \subset M$ is defined by
  \begin{equation*}
    \alpha^g (\psi_0) \definedas v^g(p) \psi_0 .
  \end{equation*}
\end{definition}
In particular, we have
\begin{equation*}
  \alpha^g ( \psi_0)
  =
  \lim_{x \to 0} \left( 
    G^g (x) \psi_0
    + \frac{1}{\omega_{n-1}|x|^n} x \cdot \psi_0
  \right).
\end{equation*}
The mass endomorphism is thus (up to a constant) defined as the zero
order term in the asymptotic expansion of the Green's function in
normal coordinates around $p$.

\subsection{Metrics flat around a point}
\label{flat_aroundp_metrics}

Let $M$ be a connected spin manifold, $p\in U$ where $U$ is an open
subset of $M$. A Riemannian metric on $U$ will be called
\emph{extendible} if it possesses a smooth extension to a (not
necessarily flat) Riemannian metric on $M$.

Fix a flat extendible metric $\gfl$ on $U$. The set of all smooth
extensions of $\gfl$ is denoted by
\begin{equation*}
  \cR_{U,\gfl}(M)
  \definedas
  \{g\,|\, \mbox{$g$ is a metric on $M$ such that
    $g|_U=\gfl$}\}.
\end{equation*}
Inside this set of metrics we study those with invertible Dirac
operator
\begin{equation*}
  \cRinv(M)
  \definedas
  \{ g \in \cR_{U,\gfl}(M) \,|\, \mbox{$D^g$ is invertible} \}.
\end{equation*}
The main subject of the article is the set
\begin{equation*}
  \cRneq(M) 
  \definedas
  \{ g \in \cRinv(M) \, | \,
  \mbox{the mass endomorphism at $p$ is not $0$} \}. 
\end{equation*}
Note that $\cR_{U,\gfl}^{\rm inv} (M)$ can be empty (see Subsection
\ref{alpha_genus}). We say that a subset $A \subset \cR_{U,\gfl} (M)$
is {\it generic} in $\cR_{U,\gfl} (M)$ if it is open in the
$C^1$-topology and dense in the $C^\infty$-topology in
$\cR_{U,\gfl}(M)$.

\subsection{The $\alpha$-genus} \label{alpha_genus}

The $\alpha$-genus is a ring homomorphism $\alpha:
\Om_*^{\mbox{spin}}(\mbox{pt}) \to KO_*(\mbox{pt})$ where
$\Om_*^{\mbox{spin}}(\mbox{pt})$ is the spin bordism ring and
$KO_*(\mbox{pt})$ is the ring of coefficients for $KO$-theory. In
particular, the well-definedness of the map means that the
$\alpha$-genus $\alpha(M)$ of a spin manifold $M$ depends only on its
spin bordism class, and the homomorphism property means that it is
additive with respect to the disjoint union and multiplicative with
respect to the product of spin manifolds. We recall that if the
dimension of $M$ is $n$ then $\alpha(M) \in KO_n(\mbox{pt})$ and as
groups we have
\begin{equation*}
  KO_n(\mbox{pt}) \cong
  \begin{cases} 
    \mZ       & \text{if $n \equiv 0 \mod 4$;}  \\
    \mZ/2\mZ  & \text{if $n \equiv 1,2 \mod 8$;}\\
    0 & \text{otherwise.}
  \end{cases}
\end{equation*}

Let $(M,g)$ be a compact spin manifold. The Atiyah-Singer index
theorem states that the Clifford index of $D^g$ coincides with
$\alpha(M)$, see \cite{lawson.michelsohn:89}. This implies that a
manifold $M$ with $\alpha(M) \neq 0$ cannot have a metric with
invertible Dirac operator. If $M$ is not connected, one can apply the
argument in each connected component. Thus there are many
non-connected examples $M$, with $\al(M) = 0$, but admitting no metric
with invertible Dirac operator.

However, the converse holds true under the additional assumption that
$M$ is connected, see \cite{ammann.dahl.humbert:09}. The proof of the
converse relies on a surgery construction preserving invertibility of
the Dirac operator together with the Stolz's examples of manifolds
with positive scalar curvature in every spin bordism class
\cite{stolz:92}, special cases were proved previously in
\cite{maier:97} and \cite{baer.dahl:02}. For our purposes, it is more
convenient to use a slightly stronger version, presented in
\cite{ammann.dahl.humbert:p09}.

\begin{theorem} 
  Let $M$ be a connected compact spin manifold and let $p\in M$. Let
  $U$ be an open subset of $M$, $p\in U\neq M$, and let $\gfl$ be a flat
  extendible metric on $U$. Then $\cR_{U,\gfl}^{\rm inv} (M) \neq
  \emptyset$ if and only if $\alpha(M) = 0$.
\end{theorem}

Using real analyticity one obtains that $\cR_{U,\gfl}^{\rm inv}(M)$ is
open and dense in $\cR_{U,\gfl}(M)$.

\subsection{Main result} \label{mainresult}

The main result of this paper is the following: If $\alpha(M) = 0$, so
that the mass endomorphism is defined for metrics in the non-empty set
$\cR_{U,\gfl}^{\rm inv}(M)$, then a generic metric has a non-zero mass
endomorphism.

\begin{theorem} \label{main} Let $M$ be a compact connected
  $n$-dimensional spin manifold with $n \geq 3$ and with vanishing
  $\alpha$-genus. Let $p \in M$ and assume that $\gfl$ is an
  extendible metric which is flat around $p$. Then there exists a
  neighborhood $U$ of $p$ for which $\cRneq(M)$ is generic in
  $\cR_{U,\gfl}(M)$.
\end{theorem}

Theorem \ref{main} will follow from Theorems~\ref{existence} and
\ref{exis_to_gen} below.

\subsection{The relation to the ADM mass}
\label{ss.adm}

Let $(M,g)$ be a compact spin manifold of dimension $n \geq 3$. Assume
that $g$ is flat in a neighborhood $U$ of a point $p \in M$. The {\em
  conformal Laplacian} is then defined by
\begin{equation*}
  L^g
  \definedas
  \frac{4(n-1)}{n-2} \Delta^g + \scal^g,
\end{equation*}
where $\Delta^g$ is the non-negative Laplacian and where $\scal^g$ is
the scalar curvature of the metric $g$. As for the Dirac operator
$D^g$, we say that a function $H^g \in L^1(M) \cap C^{\infty} (M
\setminus \{ p \})$ is the {\em Green's function} for $L^g$ if
\begin{equation*}
  L^g H^g = \delta_p
\end{equation*}
in the sense of distributions. Assume that the metric $g$ is conformal
to a metric with positive scalar curvature, then it is well known (see
for instance \cite{lee.parker:87}) that the Green's function $H^g$ of
$L^g$ exists, is positive everywhere and has the following expansion
at $p$:
\begin{equation*}
  H^g(x) = 
  \frac{1}{4(n-1)\om_{n-1}\,d^g(x,p)^{n-2}} + A^g + o(x),
\end{equation*}
where $A^g \in \mR$ and $o(x)$ is a smooth function with $o(p) = 0$.

Set $\widetilde{M} = M\setminus \{ p \}$ and $\widetilde{g} =
H^{\frac{4}{n-2}} g$. Schoen \cite{schoen:84} observed that the
complete non-compact manifold $(\widetilde{M},\widetilde{g})$ is
asymptotically flat and its ADM mass is $a_n A^g$, where $a_n > 0$
depends only on $n$. We recall that an asymptotically flat manifold,
if interpreted as a time symmetric spacelike hypersurface of a
lorentzian manifold, is obtained by considering an isolated system at
a fixed time in general relativity. The ADM mass gives the total
energy of this system. With this remark, the number $A^g$ is often
called the mass of the compact manifold $(M,g)$. By analogy, the
operator $\al^g(p)$, which is by construction the spin analog of
$A^g$, is called the "mass endomorphism" of $(M,g)$ at $p$. We will
also see in Subsection \ref{ss.conclus} that the mass
endomorphism plays the same role as the number $A^g$ in a Dirac
operator version of the Yamabe problem.

\subsection{Conclusions of non-zero mass}
\label{ss.conclus}

In this Subsection we will summarize why we are interested in metrics
with non-zero mass endomorphism.

Let $(M,g)$ be a compact Riemannian spin manifold of dimension $n \geq
2$. For a metric $\widetilde g$ in the conformal class $[g]$ of $g$,
let $\lambda_1(\widetilde{g})$ be the eigenvalue of the Dirac operator
$D^g$ with the smallest absolute value (it may be either positive or
negative). We define
\begin{equation*}
  \lamin(M,[g]) 
  = 
  \inf_{\widetilde{g} \in [g] } |\lambda_1(\widetilde{g})|
  \Vol^{\widetilde{g}} (M)^{1/n}. 
\end{equation*}
For this conformal invariant $\lamin(M,[g])$ it was proven in
\cite{ammann:03,ammann:habil} and
\cite{ammann.grosjean.humbert.morel:08} that
\begin{equation*}
  0 < 
  \lamin(M,[g]) \leq \lamin(\mS^n) 
  = 
  \frac{n}{2}\, \om_n^{1/n}.
\end{equation*}
The strict inequality
\begin{equation} \label{str_in} \lamin(M,[g]) < \frac{n}{2}
  \om_n^{1/n}
\end{equation}    
has several applications, see
\cite{ammann:09,ammann.grosjean.humbert.morel:08,ammann.humbert.morel:06}:

\begin{itemize}
\item Inequality \eref{str_in} implies that the invariant
  $\lamin(M,[g])$ is attained by a generalized metric, that is, a
  metric of the form $|f|^{2/(n-1)} g$ where $f \in C^2(M)$ can have
  some zeros;
\item Inequality \eref{str_in} gives a solution of a conformally
  invariant partial differential equation which can be read as a
  nonlinear eigenvalue equation for the Dirac operator, a type of
  Yamabe problem for the Dirac operator;
\item using Hijazi's inequality \cite{hijazi:91} one obtains a
  solution of the standard Yamabe problem which consists of finding a
  metric with constant scalar curvature in the conformal class of $g$
  in the case of $n \geq 3$.
\end{itemize}

The first two applications can be interpreted as a spin analog of the
Yamabe problem for many reasons, see \cite{ammann:03}. The third
application says that a non-zero mass endomorphism can be used in the
Yamabe problem instead of the positivity of the mass $A^g$ defined in
Subsection~\ref{ss.adm}.
 
Now, let us come back to the subject of this paper. In
\cite{ammann.humbert.morel:06}, we prove that a non-zero mass
endomorphism implies Inequality \eref{str_in}. In particular we see
with Theorem \ref{main} that Inequality \eref{str_in} holds for
generic metric in $\cR_{U,\gfl}(M)$. As a consequence, for generic
metrics in $\cR_{U,\gfl}(M)$, we have all the applications stated
above.

This can be compared to the Yamabe problem: Schoen proved that the
positivity of the number $A^g$, that is the mass of $(M,g)$ defined in
Subsection~\ref{ss.adm}, implies a solution of the standard Yamabe
problem. The positive mass theorem implies that $A^g \geq 0$. Hence, we
get a solution of the Yamabe problem as soon as $A^g \neq 0$. In
particular, the mass endomorphism plays the same role in the Yamabe
problem for the Dirac operator as the mass in the classical Yamabe
problem.

\subsection{Further remarks}

We here discuss extensions of the results in this paper. At first we
ask what can be done without the condition of flatness in a
neighborhood of $p$. For an arbitrary metric on $M$ one possible
extension of our setup is a relative version of the mass endomorphism.

To briefly sketch this relative version, assume that there is a
manifold $(M',g')$ and assume that a point $p'$ has a neighborhood
which is orientation preserving isometric to a neighborhood of $p$ in
$(M,g)$. Using this isometry the difference between the Green's
function $G^g_p$ of $D^g$ on $M$ and the Green's function
$G^{g'}_{p'}$ of $D^{g'}$ on $M'$ is a well-defined smooth spinor in a
neighborhood of $p \cong p'$. Then the relative mass endomorphism is
defined as $G_p(p)-G_{p'}(p') \in \End(\Si_pM) \cong
\End(\Si_{p'}M')$.  The methods of the present work can be modified
such that this mass endomorphism is non-zero for generic metrics $g$
on $M$ which are locally isometric to a fixed metric $g'$ on $M'$
around $p$ and $p'$.

Now we discuss whether the condition $\al(M)=0$ is necessary. If the
manifold $M$ has a non-trivial index, then $\cRinv(M)$ is empty.
Nevertheless an extension is possible. For this $\cRinv(M)$ has to be
replaced by the space of metrics for which the kernel of the Dirac
operator has minimal dimension. For such metrics there are various
choices of ``Green's functions'' for which the mass endomorphism is
generically non-zero, for example if one defines it as being the
integral kernel of the operator $(D+\pi)^{-1}-\pi$ where $\pi$ is the
projection to the kernel.

In \cite{ammann.dahl.hermann.humbert:p10b} we plan to present another
method to prove a variant of Theorem~\ref{mass_to_infty} with slightly 
different conditions and a different potential for
generalization. This other proof uses methods from spectral theory,
and explains that the convergence to infinity of the mass endomorphism
actually can be understood as a pole of a meromorphic function.

\subsection{Overview of the paper}

We here give a short overview of the paper. In Section
\ref{sec_notation} we introduce notation and collect basic facts
concerning spinors and Dirac operators. In Section
\ref{section_existence} we explain how to find one metric with
non-zero mass endomorphism on a given manifold, this uses the results
of the following two sections. In Section \ref{section_mass_to_infty}
we show that under certain assumptions the mass endomorphism tends to
infinity when the Riemannian metric varies and approaches a metric
with harmonic spinors. In Section \ref{section_surgery} we show that
the property of non-zero mass endomorphism can be preserved under
surgery on the underlying manifold. Finally, in Section
\ref{section_exis_to_gen} we use analytic perturbation techniques to
show that the existence of one metric with non-zero mass endomorphism
implies that a generic metric has this property.

\section{Notations and preliminaries}
\label{sec_notation}

\subsection{Notation and some basic facts}

In this article we use the following notations for balls and spheres:
$B^k(R) \definedas \{ x \in \mR^k \,|\, \|x\| < R \}$, $B^k \definedas
B^k(1)$, $S^k(R) \definedas \{ x \in \mR^k \,|\, \|x\| = R \}$, $S^k
\definedas S^k(1)$.

As background for basic facts on spinors and Dirac operators we refer
to \cite{lawson.michelsohn:89} and \cite{friedrich:00}. For the
convenience of the reader we summarize a few definitions and facts.
On a compact Riemannian spin manifold $(M,g)$ one defines the Dirac
operator $D^g$ acting on sections of the spinor bundle. The Dirac
operator is essentially self-adjoint and extends to a self-adjoint
operator $H^1\to L^2$ where $H^1$ is the space of $L^2$-spinors whose
first derivative is $L^2$ as well, and $L^2$ is the space of square
integrable spinors. A smooth spinor is called \emph{harmonic}, if it
is in the kernel of the Dirac operator $D^g$.  Any $L^2$-spinor
satisfying $D^g \phi=0$ in the weak sense, is already smooth, thus it
is a harmonic spinor. If the kernel of $D^g$ is trivial, then the
Dirac operator is invertible with a bounded inverse $L^2 \to H^1$. The
inverse has an integral kernel called the Green's function of $D^g$.
The Green's function of $D^g$ was already used in
Subsection~\ref{submassendo} to define the mass endomorphism.

\subsection{Comparing spinors for different metrics}
\label{identification_section}

Let $g$ and $h$ be Riemannian metrics on the spin manifold $M$. The
goal of this section is to recall how spinors on $(M,g)$ are
identified with spinors on $(M,h)$ using the method of Bourguignon and
Gauduchon \cite{bourguignon.gauduchon:92}, see also
\cite{ammann.dahl.humbert:09}.

Given the metrics $g$ and $h$ there exists a unique 
bundle endomorphism $a^g_h$ of $TM$ which satisfies 
$g(X,Y)=h(a^g_h X, Y)$ for all $X$, $Y \in TM$. It is 
$g$-self-adjoint and positive definite. Define $b^g_h:=
(a^g_h)^{-1/2}$, 
where $(a^g_h)^{1/2}$ is the unique positive pointwise 
square root of $a^g_h$.
The map $b^g_h$ maps $g$-orthonormal frames to $h$-orthonormal
frames and defines an $\SO(n)$-equivariant bundle morphism
$b^g_h: \SO(M,g) \rightarrow \SO(M,h)$ of the principal bundles of
orthonormal frames. The map $b^g_h$ lifts to a $\Spin(n)$-equivariant
bundle morphism $\beta^g_h:\Spin(M,g) \rightarrow \Spin(M,h)$ of the
corresponding spin structures. From this we obtain a homomorphism of
vector bundles
\begin{equation} \label{bourg-map} 
\be^g_h: \Si^g M \to \Si^h M
\end{equation}
which is a fiberwise isometry with respect to the inner products on
$\Si^g M$ and $\Si^h M$. We let the Dirac operator $D^h$ act on
sections of $\Si^g M$ by defining 
\begin{equation*}
  D^h_g
  \definedas
  (\beta^g_h)^{-1} D^{h} \beta^g_h .
\end{equation*}
In \cite[Thm. 20]{bourguignon.gauduchon:92} an expression for $D^h_g$ 
is computed in terms of a local $g$-orthonormal frame $\{ e_i \}_{i=1}^n$. 
The result is
\begin{equation} \label{D^h_g} 
  D^h_g \varphi = \sum_{i=1}^n e_i \cdot \nabla^g_{b^g_{h}(e_i)} \varphi
  +\frac{1}{2}\,\sum_{i=1}^n e_i \cdot 
  ( (b^g_h)^{-1} \nabla^h_{b^g_{h}(e_i)} b^g_h 
  - \nabla^g_{b^g_{h}(e_i)} ) \cdot \varphi ,
\end{equation}
where for any vector field $X$ the operator 
$(b^g_h)^{-1} \nabla^h_X b^g_h - \nabla^g_X$ 
is $g$-antisymmetric and therefore 
considered as an element of the Clifford algebra.
It follows that
\begin{equation} \label{D^h_g_estimate}
  D^h_g \phi
  =
  D^{g} \phi + A^h_g (\nabla^g \phi ) + B^h_g(\phi),
\end{equation}
where $A^h_g$ and $B^h_g$ are pointwise vector bundle maps whose
pointwise norms are bounded by $C |h - g|_g$ and $C (|h - g|_g +
|\nabla^g(h - g)|_g )$ respectively.

\section{Finding one metric with non-vanishing mass endomorphism}
\label{section_existence}

The goal of this section is to prove the following Theorem.

\begin{theorem} \label{existence} Let $M$ be a compact connected spin
  manifold of dimension $n \geq 3$ and let $p \in M$. Assume that
  $\alpha(M) = 0$. Then there exists a neighborhood $U$ of $p$ and a
  flat metric $\gfl$ on $U$ such that $\cRneq(M)$ is non-empty.
\end{theorem}
 
\begin{proof}
  We start by proving the theorem when the manifold is a torus.
  Consider the torus $T^n$ equipped with the Lie group spin structure
  for which the standard flat metric $g_0$ has a space of parallel
  spinors of maximal dimension. Choose $p \in T^n$ and let $U$ be a
  small open neighborhood of $p$. Further, let $\gfl$ be the
  restriction of $g_0$ to $U$.

  Since $n \geq 3$ we have that $\alpha(T^n) = 0$ so by
  \cite{ammann.dahl.humbert:09} there is a metric $g_1$ on $T^n$ with
  invertible Dirac operator. The construction of $g_1$ is done through
  a sequence of surgeries which starts with the disjoint union of
  $T^n$ and some other manifolds, and ends with the torus $T^n$. These
  surgeries can be arranged so that they do not change the open set
  $U$ in the initial $T^n$, so the resulting metric satisfies $g_1 =
  g_0$ on $U$, or $g_1 \in \cRinv(T^n)$.

  Define the family of metrics $g_t \definedas t g_1 + (1-t)g_0$.
  Since the eigenvalues of $D^{g_t}$ depend analytically on $t$ it
  follows that $D^{g_t}$ is invertible except for isolated values of
  $t$, it follows that $g_t \in \cRinv(T^n)$ except for isolated values
  of $t$. Choose a sequence $t_k \to 0$ for which $g_{t_k} \in
  \cRinv(T^n)$, we can then apply Theorem \ref{mass_to_infty} below to
  the sequence $g_{t_k}$ converging to $g_0$ and conclude that
  $g_{t_k} \in \cRneq(T^n)$ for $k$ large enough. In particular
  $\cRneq(T^n)$ is not empty, and we choose a metric $h_0$ from this
  set.

  Now let $M$ be a manifold of dimension $n$ as in the theorem. Since
  $\alpha(M) = 0$ we know that there is a metric $g$ on $M$ with
  invertible Dirac operator. We consider the disjoint union
  \begin{equation*}
    M_0 = T^n \sqcup (-T^n) \sqcup M.
  \end{equation*}
  Here $-T^n$ denotes $T^n$ with the opposite orientation, so that
  $T^n \sqcup (-T^n)$ is a spin boundary and $M_0$ is spin bordant to
  $M$. Since $M$ is connected it follows that $M$ can be obtained from
  $M_0$ by a sequence of surgeries of codimension $2$ and higher, see
  \cite[Proposition 4.3]{ammann.dahl.humbert:09}. Again, these
  surgeries can be arranged to miss the open set $U$ in the first
  $T^n$. We equip $M_0$ with the Riemannian metric $h_0 \sqcup h_0
  \sqcup g \in \cRneq(T^n \sqcup (-T^n) \sqcup M)$ and when we use
  Theorem \ref{conv_mass} below for the sequence of surgeries we end
  up with a metric $g' \in \cRneq(M)$.

  Finally, the point $p \in M$ we end up with after the sequence of
  surgeries might of course not be equal to the point $p$ in the
  assumptions of the theorem. If we set this right by a diffeomorphism
  we have proved that $\cRneq(M)$ is non-empty.
\end{proof}

Note that this proof does not work in dimension $2$. Indeed, we
strongly use that the $\alpha$-genus of the torus $T^n$ vanishes. This
fact is only true in dimension $n\geq 3$. If the flat torus $T^2$ is
equipped with the Lie group spin structure with two parallel spinors,
then $\alpha(T^2) = 1$. By the way, it is proven in
\cite{ammann.humbert.morel:06} that the mass endomorphism always
vanishes in dimension $2$.

\section{Mass endomorphism of metrics close to a metric with harmonic
  spinors}
\label{section_mass_to_infty}

Finding examples of metrics with non-zero mass endomorphism seems to
be a difficult issue. The only explicit examples we have until now are
the projective spaces $\mR P^n$, $n \equiv 3 \mod 4$, equipped with its
standard metric, see \cite{ammann.humbert.morel:06}. The goal of this
section is to show that metrics $g \in \cRinv (M)$ sufficiently close
to a metric $h \in \cR_{p,U,\gfl} \setminus \cRinv(M)$ will under some
additional assumptions provide such examples. This is the object of
Theorem \ref{mass_to_infty} below, which in our mind has an interest
independently of the application to Theorem \ref{main}.

\begin{theorem} \label{mass_to_infty} Let $U$ be a neighborhood of
  $p\in M$. Assume that $h \in \cR_{U,\gfl}(M)$ has $\ker D^h \neq
  \{0\}$. Further assume that the evaluation map of harmonic spinors
  at $p$,
  \begin{equation*}
    \ker D^h \ni \psi \mapsto \psi(p) \in \Sigma^h_p M,
  \end{equation*}
  is injective. Set $m \definedas \dim$ $\ker D^h$ Let $g_k \in
  \cRinv(M)$, $k = 1,2,\dots$, be a family of metrics on $M$
  converging to $h$ in the $C^1$-topology.

  Then the mass endomorphism $\alpha^{g_k}$ at $p$ has at least $m$
  eigenvalues tending to $\infty$ as $k \to \infty$. In particular,
  $g_k \in \cRneq(M)$ for large $k$.
\end{theorem}

The proof of this theorem is inspired by the work of Beig and
O'~Murchadha \cite{beig.omurchadha:91}. In the hypothesis of Theorem
\ref{mass_to_infty}, the injectivity of the evaluation map $\ker D^h
\ni \psi \mapsto \psi(p) \in \Sigma^h_p M,$ is quite restrictive: it
is fulfilled for instance when the space of harmonic spinors is
$1$-dimensional if $p$ is not a zero of the harmonic spinor. In
Theorem \ref{existence} we applied the result to the flat torus $T^n$.

\begin{proof}
  For the proof we choose a non-zero $\psi\in \ker D^h$. Set $\psi_p
  \definedas \psi(p) \in \Sigma^h_p M$, by assumption we have $\psi_p
  \neq 0$. We will show that $\al^{g_k}(\psi_p)$ tends to infinity.

  Let $G_k$ be the Green's function of $D^{g_k}$ associated to
  $\psi_p$, that is $G_k$ is a distributional solution of
  \begin{equation*}
    D^{g_k} G_k =  \de_p \psi_p.
  \end{equation*}
  In coordinates around $p$ we write (compare Proposition
  \ref{prop.before.mass})
  \begin{equation} \label{formel.green} G_k = -\eta \frac{x}{\om_{n-1}
      r^n} \cdot \psi_p + v^{g_k} (\psi_p).
  \end{equation}
  Here $\eta$ is a cutoff function which is equal to $1$ near $p$ and
  has support in $U$. We shorten notation by writing $v_k$ for the
  spinor field $v^{g_k}( \psi_p)$.

  {\bf Step 1.} {\em We show that there are $p_k \in M$ for which
    $|v_k(p_k)| \to \infty$.} Let the smooth function $\Omega: M
  \setminus \{p\} \to (0,1]$ satisfy
  \begin{equation*}
    \Omega(x) =
    \begin{cases}
      r(x)    & \text{if $x \in B_p(\epsilon)$},\\
      1 & \text{if $x \in M \setminus B_p(2\epsilon)$ }.
    \end{cases}
  \end{equation*}
  Note that $\Omega$ does not depend on $k$. We have
  \begin{equation*}
    \begin{split}
      0 < |\psi_p|^2 &=
      \int_M \<G_k, D^{g_k} \psi \> \, dv^{g_k} \\
      &= \int_M \frac{1}{\Omega^{n-1}}
      \<\Omega^{n-1} G_k, D^{g_k} \psi \> \, dv^{g_k} \\
      &\leq \int_M \frac{1}{\Omega^{n-1}} \, dv^{g_k} \| \Omega^{n-1}
      G_k \|_{\infty} \| D^{g_k} \psi \|_{\infty}.
    \end{split}
  \end{equation*}
  As the integral is bounded and the last factor tends to zero as $k
  \to \infty$, we conclude that
  \begin{equation*}
    \lim_{k \to \infty} \| \Omega^{n-1} G_k \|_{\infty} = \infty.
  \end{equation*}
  Let $p_k$ be points for which
  \begin{equation*}
    |\Omega^{n-1}(p_k) G_k(p_k)| = \| \Omega^{n-1} G_k \|_{\infty} .
  \end{equation*}
  Then
  \begin{equation*}
    \Omega^{n-1}(p_k) G_k(p_k) 
    = 
    \Omega^{n-1}(p_k) \left( 
      -\eta \frac{x}{\om_{n-1} r^n} \cdot \psi_0
    \right) (p_k)
    + 
    \Omega^{n-1}(p_k)v_k(p_k), 
  \end{equation*}
  here the first term on the right hand side is bounded so the second
  term must tend to infinity. Since $|\Omega^{n-1}(p_k)v_k(p_k)| \leq
  |v_k(p_k)|$ we conclude that $|v_k(p_k)| \to \infty$ as $k \to
  \infty$, and Step~1 is proven.

  To the spinor $v_k$ which is a section of $\Si^{g_k} M$ the map
  $\be^{g_k}_h$ described in (\ref{bourg-map}) associates a section
  $w_k \definedas \be^{g_k}_h v_k$ in the spinor bundle $\Si^{h} M$.
  We decompose this section as
  \begin{equation*}
    w_k 
    =
    a_k \phi_k + w_k^{\perp}  
  \end{equation*}
  where $\phi_k \in \ker D^h$ is normalized to have $\| \phi_k
  \|_{L^p(\Si^h M)} = 1$ , $a_k \in \mR$, and $w_k^{\perp}$ is
  orthogonal to $\ker D^h$. We choose $p$ large enough so that
  $H_1^p(\Si^h M)$ embeds into $C^0(\Si^h M)$.

  {\bf Step 2.} {\em We show that $|a_k| \to \infty$.}  For a
  contradiction assume that the sequence $|a_k|$ is bounded.  {}From
  (\ref{formel.green}) it follows that $D^{g_k} v_k = \grad \eta
  \cdot \frac{x}{\om_{n-1} r^n}\cdot \psi_p$.
  This together with the
  properties of $\be^{g_k}_h$ gives
  \begin{equation} \label{alpha_perp_estimate}
    \begin{split}
      \| w_k^{\perp} \|_{H_1^p} &\leq
      C \| D^h w_k^{\perp} \|_{L^p} \\
      &=
      C \| D^h w_k \|_{L^p} \\
      &=
      C \| (\be^{g_k}_h)^{-1} D^h \be^{g_k}_h v_k \|_{L^p} \\
      &=
      C \| D^h_{g_k} v_k \|_{L^p} \\
      &\leq C \| D^{g_k} v_k \|_{L^p}
      + C \| A^h_{g_k} (\nabla^{g_k} v_k ) + B^h_{g_k}( v_k ) \|_{L^p} \\
      &\leq C \| \grad \eta \cdot \frac{x}{\om_{n-1} r^n}\cdot \psi_p
      \|_{L^p} + C \epsilon_k \| w_k \|_{H_1^p},
    \end{split}
  \end{equation}
  here the first term is bounded and $\epsilon_k \to 0$ by our
  assumption that $g_k \to h$ in the $C^1$-topology. By assumption we
  also have
  \begin{equation*}
    \begin{split}
      \| w_k \|_{H_1^p} &\leq \| a_k \phi_k \|_{H_1^p} +
      \| w_k^{\perp} \|_{H_1^p} \\
      &\leq
      C + \| w_k^{\perp} \|_{H_1^p}. \\
    \end{split}
  \end{equation*}
  Together this gives
  \begin{equation*}
    \| w_k^{\perp} \|_{H_1^p}
    \leq
    C + C \epsilon_k  + C \epsilon_k \| w_k^{\perp} \|_{H_1^p},
  \end{equation*}
  so $\| w_k^{\perp} \|_{H_1^p}$ is bounded. We conclude that $\|
  w_k^{\perp} \|_{C^0}$ is bounded, and the assumption that $|a_k|$ is
  bounded then tells us that $\| w_k \|_{C^0} = \| v_k \|_{C^0}$ is
  bounded. This contradicts Step~1, so we have proved Step~2.

  {\bf Step 3.} {\em Conclusion.}  Set $\omega_k \definedas a_k^{-1}
  w_k $ and $\omega_k^{\perp} \definedas a_k^{-1} w_k^{\perp} $ so
  that
  \begin{equation*}
    \omega_k = \phi_k + \omega_k^{\perp}.
  \end{equation*}
  Then (\ref{alpha_perp_estimate}) tells us that
  \begin{equation*}
    \| \omega_k^{\perp} \|_{H_1^p}
    \leq
    C a_k^{-1}
    \| \grad \eta \cdot \frac{x}{\om_{n-1} r^n} \cdot \psi_0 \|_{L^p}
    +
    C \epsilon_k \| \omega_k \|_{H_1^p},
  \end{equation*}
  where the first term now tends to zero. Since the $\phi_k$ are in
  $\ker D^h$ and they are normalized in $L^p(\Si^h M)$ it follows that
  they are bounded in $H_1^p(\Si^h M)$. From this we get
  \begin{equation*}
    \begin{split}
      \| \omega_k \|_{H_1^p} &\leq \| \phi_k \|_{H_1^p} +
      \| \omega_k^{\perp} \|_{H_1^p} \\
      &\leq C +
      \| \omega_k^{\perp} \|_{H_1^p} . \\
    \end{split}
  \end{equation*}
  It follows that
  \begin{equation*}
    \| \omega_k^{\perp} \|_{H_1^p}
    \leq
    o(1) + C\epsilon_k \| \omega_k^{\perp} \|_{H_1^p} 
  \end{equation*}
  so $\| \omega_k^{\perp} \|_{H_1^p} \to 0$ and $\| \omega_k^{\perp}
  \|_{C^0} \to 0$. Finally we have
  \begin{equation*}
    \begin{split}
      |\al^{g_k} (\psi_p)| &=
      |v_k(p)| \\
      &=
      |w_k(p)| \\
      &=
      a_k |\omega_k(p)| \\
      &\geq
      a_k ( |\phi_k (p)| - | \omega_k^{\perp}(p)|) \\
      &=
      a_k ( |\phi_k (p)| + o(1) ). \\
    \end{split}
  \end{equation*}
  By our assumption that the evaluation map of harmonic spinors at $p$
  is injective we know that $|\phi_k (p)|$ cannot tend to zero, so
  from Step~2 we conclude that $|\al^{g_k} (\psi_p)| \to \infty$. This
  finishes the proof of Step~3 and the Theorem.
 
\end{proof}

\section{Surgery and non-zero mass endomorphism}
\label{section_surgery}

Let $\wihat M$ be obtained from $M$ by surgery of codimension at least
$2$. We assume that $p\in M$ is not hit by the surgery, so we have
$p\in \wihat M$. As before $\cRneq(M)$ denotes the metrics with
invertible Dirac operator on $M$ which coincide with the flat metric
$\gfl$ on $U$ and whose mass endomorphism at $p$ is not zero.  The
goal of this section is to prove that $\cRneq(M)\neq \emptyset$
implies $\cRneq(\wihat M)\neq \emptyset$.

We start with a manifold $M$ of dimension $n$ and a point $p \in M$.
We will perform a surgery of dimension $k \in \{ 0,\cdots n-2\}$ on
$M$. For this construction, we follow the beginning of Section 3 in
\cite{ammann.dahl.humbert:09} and use the same notation. So, we assume
that we have an embedding $i: S^k \to M$ with a trivialization of the
normal bundle of $S \definedas i(S^k)$ in $M$, which thus can be
identified with $S^k \times \mR^{n-k}$. The normal exponential map
then defines an embedding of a neighborhood of the zero section of the
normal bundle of $S$, in other words for small $R>0$ the normal
exponential map defines a diffeomorphism $f$ from $S^k\times
B^{n-k}(R)$ to an open neighborhood of $S$, and $f$ is an extension of
$S^k\times\{0\}\to S^k \stackrel{i}{\to} M$. Furthermore, for
sufficiently small $R>0$, the distance from $f(x,y)$ to $S =
f(S^k\times \{0\})$ is $|y|$.

As before we assume that $U$ is an open neighborhood of $p$, on which
a flat extendible metric $\gfl$ exists. We assume further that $p
\not\in S$, and by possibly restricting $U$ to a smaller open set, we
can also assume that $\overline U \cap S = \emptyset$. Thus for small
$R>0$ one obtains
\begin{equation*}
  U \cap f(S^k\times \overline{B^{n-k}(R)}) = \emptyset.
\end{equation*}
As in Section 1 of \cite{ammann.dahl.humbert:09} we define
\begin{equation*}
  \wihat{M} 
  = 
  \left( M \setminus f(S^k \times \overline{B^{n-k}(R)}) \right) 
  \cup \left(\overline{B^{k+1}}\times S^{n-k-1}\right)/{\sim},
\end{equation*}
where $\sim$ identifies the boundary of $\overline{B^{k+1}} \times
S^{n-k-1}$ with $f(S^k \times S^{n-k-1}(R))$ via the map $(x,y)
\mapsto f(x,Ry)$. Our constructions are carried out such that $U$ is
both a subset of $M$ and $\wihat M$.

The main result of this section is the following Theorem.
\begin{theorem}\label{conv_mass}
  If $\cRneq(M)\neq \emptyset$, then $\cRneq(\wihat M) \neq
  \emptyset$.
\end{theorem}

%
%
%
%

\begin{proof}
  We assume the requirements for $p$, $U$, $f$ and $k$ stated at the
  beginning of this section, and let $g\in \cRneq(M)$. The goal is to
  construct a metric $\wihat{g}\in\cRneq(\wihat M)$ following the
  constructions in \cite{ammann.dahl.humbert:09}.

  Theorem 1.2 in \cite{ammann.dahl.humbert:09} allows us to construct
  a metric $\wihat{g}'$ on $\wihat{M}$ with invertible Dirac operator.
  We recall the scheme of the proof of this theorem.  As in the
  beginning of Section 3 of \cite{ammann.dahl.humbert:09} we define
  open neighborhoods $U_S(r)$ by
  \begin{equation*}
    U_S(r)
    \definedas
    f (S^k \times B^{n-k}(r))
  \end{equation*}
  for small $r$. Then we construct a family of metrics $(g_\rho)_\rho$
  satisfying $g_\rho = g$ on $M\setminus U_S(R_{\rm{max}})$ for some
  small number $R_{\rm{max}}$. This family of metrics is constructed
  in two steps. First, we use Proposition 3.2 in
  \cite{ammann.dahl.humbert:09} to assume that $g$ has a product form
  in a neighborhood of $S$. Then, we do the construction of Section
  3.2 in \cite{ammann.dahl.humbert:09} to get $g_\rho$. Once these
  metrics $(g_\rho)$ are constructed, we proceed by contradiction. We
  take a sequence $(\rho_k)_{k\in \mN}$ tending to $0$ and we assume
  that $\ker(D^{g_{\rho_k}} ) \neq 0$ for all $k$, that is
  \begin{equation} \label{assumpH1} \forall k \in \mN, \, \hbox{ there
      exists a harmonic spinor } \psi_k \neq 0 \hbox{ on }
    (\wihat{M},g_{\rho_k}) .
  \end{equation} 
  By showing that $\lim_{k\to\infty} \psi_k$ converges in a weak sense
  to a non-zero limit spinor in $\ker D^g$, we will obtain a
  contradiction. So the metric $\wihat{g'} \definedas g_\rho$
  satisfies the requirements of Theorem 1.2 in
  \cite{ammann.dahl.humbert:09} as soon as $\rho$ is small enough.

  This proof actually allows us to require an additional property for
  the metrics $g_\delta$, and make weaker assumptions on the spinors
  $\psi_k$.
  \begin{itemize}
  \item The number $R_{\rm{max}}$ in the proof can be chosen
    arbitrarily small. So set $\de=R_{\rm{max}}$ and choose $\rho
    \definedas \rho(\de)$ small enough so that $g_\de = g_{\rho}$ has
    an invertible Dirac operator. We obtain in this way a family of
    metrics $(g_\de)_{\de \in (0,\de_0)}$ for some $\de_0>0$ such that
    all $D^{g_\de}$ are invertible and such that $g_\de = g$ on $M
    \setminus U_S(\de)$.
  \item Let now $(\de_k)_{k \in \mN}$ be a sequence of positive
    numbers going to $0$. We make the following assumption:
    \begin{equation*}
      \begin{array}{c} 
        \forall k \in \mN, \, \hbox{there exists a spinor } \psi_k \hbox{ on }
        (\wihat{M},g_{\de_k}) 
        \hbox{ and a sequence } \\
        \la_k \hbox{ converging to } 0 \hbox{ such that }
        D^{g_{\de_k}}  \psi_k = \la_k \psi_k.  
      \end{array}
    \end{equation*}
    Working with these spinors instead of the ones given by assumption
    \eref{assumpH1}, the same contradiction is obtained. This proves
    that there is a uniform spectral gap for
    $(g_\de)_{\de\in(0,\de_0/2)}$, or in other words that there exists
    a constant $C_0 > 0$ independent of $\de \in (0,\de_0/2)$ such
    that
    \begin{equation} \label{sgap} \spec{D^{g_{\delta}}} \cap
      [-C_0,C_0] = \emptyset.
    \end{equation}   
  \end{itemize}
 
  Now, we prove that the metric $\wihat{g} \definedas g_\de$ for
  $\delta$ small enough satisfies the requirements of
  Theorem~\ref{conv_mass}. It is already clear that $D^{g_\de}$ is
  invertible for $\de$ small enough, and that $g_\de$ is flat on $U$
  for $\de$ small enough. It remains to show that $\al_p^{g_\de} \neq
  0$ for $\de$ small enough. For this purpose we show that
  $\al_p^{g_\de} \to \al_p^g$ as $\de \to 0$. Since we assume $\al_p^g
  \neq 0$ this gives the desired result.

  So let us prove this fact. First, choose $\psi_0 \in \Sigma_p^g(M)=
  \Sigma_p^{g_\de}(M)$. To simplify the notation, set $\gamma
  \definedas G^{g} \psi_0$ and $\gamma_{\de} \definedas G^{g_{\de}}
  \psi_0$. The proof will be complete if we prove that
  \begin{equation} \label{diff_mass} \lim_{\de \to 0} \gamma(p) -
    \gamma_{\de}(p) = 0.
  \end{equation} 
  Note that the spinor $\gamma - \gamma_{\de}$, defined on $M\setminus
  (\{p\}\cup U_S(\delta))$, is smooth and extends smoothly to $p$.
  Indeed, it is equal on $U$ to $ v_p^g(x) \psi_0 - v_p^{g_\de}(x)
  \psi_0$ (with the notations of Proposition \ref{prop.before.mass}
  and Definition \ref{def.mass}). Let $\eta_{\de} \in
  C^{\infty}(\wihat{M})$, $0 \leq \eta_\de \leq 1$ be a cut-off
  function such that $\eta_\de = 1$ on $M \setminus U_S(3 \de)$ and
  $\eta_\de = 0$ on $U_S(2 \de)$. Since on $\supp(\eta_\de) \subset
  \wihat{M} \setminus U_S(2 \de) = M \setminus U_S(2 \delta)$ we have
  $g_\de = g$ we may assume that
  \begin{equation} \label{nablaeta} |d \eta_\de|_g =
    |d\eta_\de|_{g_\de} \leq \frac{2}{\de}.
  \end{equation}  
  {}From Equation (\ref{sgap}), we have
  \begin{equation*}
    C_0^2 
    \leq 
    \frac{ \int_{\wihat{M} } |D^{g_\de } \phi_\de|^2_{g_\de} dv^{g_\de}}
    {\int_{\wihat{M} } |\phi_\de|^2_{g_\de} dv^{g_\de}} 
  \end{equation*}
  for all smooth non-zero spinors $\phi_\de$ on $(\wihat{M},g_\de)$.
  We evaluate this quotient for $\phi_\de \definedas \eta_\de \gamma -
  \gamma_\de$.  Note that $\phi_\de$ is well defined on $(\wihat
  M,g_\de)$ and smooth since $\gamma$ is well defined on
  $\supp(\eta_\de)$.  Since $\gamma$ and $\gamma_\de$ are harmonic, we
  have $D\phi_\de =d\eta_\de \cdot \gamma$, and since $g_\de= g$ on
  $\supp(\eta_\de)$, we get from Equation (\ref{nablaeta}) that
  \begin{equation*}
    \begin{split} 
      \int_{\wihat{M} } |D^{g_\de } \phi_\de|^2_{g_\de} dv^{g_\de} &=
      \int_{\wihat{M} } |d\eta|^2_g |\gamma|_g^2 dv^g\\
      &\leq \frac{4}{\de^2} \sup_{x \in U_S(3 \delta_0)}\left(
        |\gamma(x)|^2 \right) \Vol^g \left(U_S(3 \de) \setminus
        U_S(2\de)\right).
    \end{split}
  \end{equation*}
  We have that $\Vol^g\left(U_S(3 \de) \setminus U_S(2\de)\right) \leq
  C\de^{n-k}$ where we used the convention (used throughout this
  proof) that $C$ is a positive constant independent of $\de$. Since
  $k\leq n-2$, this leads to
  \begin{equation*} 
    \int_{\wihat{M} } |D^{g_\de } \phi_\de|^2_{g_\de} dv^{g_\de}  \leq C.
  \end{equation*}
  Since $\eta_\de = 1 $ on $M \setminus U_S(3\de)$ and since $g_\de =
  g$ on this set, it follows that
  \begin{equation} \label{den1} \int_{M \setminus U_S(3\de)}
    |\phi_\de|^2_{g_\de} dv^{g_\de} \leq C .
  \end{equation} 

  Now, we proceed as in step 2 of the proof of Theorem 1.2 in
  \cite{ammann.dahl.humbert:09}. Let $Z > 0$ be a large integer.  By
  (\ref{den1}) the set $\{ \phi_\de \}_{\de >0}$ is bounded in $L^2(M
  \setminus U_S(1/Z))$. By Lemma 2.2 in \cite{ammann.dahl.humbert:09}
  it follows that $\{ \phi_\de \}_{\de >0}$ is bounded in $C^{1,\al}(M
  \setminus U_S(2/Z))$ for all $\al$. We apply Ascoli's Theorem and
  conclude there is a subsequence $(\phi_{\de_k})$ of $\{ \phi_\de
  \}_{\de >0 }$ which converges in $C^1(M \setminus U_S(2/Z))$ to a
  spinor $\Phi_0$. Similarly we construct further and further
  subsequences of $(\phi_{\de_k})$ converging to $\Phi_i$ in $C^1(M
  \setminus U_S(2/(Z + i)))$. Taking a diagonal subsequence of these
  subsequences, we obtain a subsequence $(\phi_{\de_k})$ which
  converges in $C^1_{\rm{loc}}(M \setminus S)$ to a spinor $\Phi$. As
  $\phi_{\de}$ is $D^g$-harmonic on $(M\setminus U_S(3\de))$ the
  $C^1_{\rm{loc}}(M \setminus S)$-convergence implies that $D^g \Phi =
  0$ on $M\setminus S$.  With (\ref{den1}) we conclude that $\Phi \in
  L^2 (M)$.  Thus $\Phi$ is $L^2$ and smooth on $M\setminus S$. The
  equation $D^g\Phi=0$ holds on $M\setminus S$. We now apply Lemmas
  2.1 and 2.4 of \cite{ammann.dahl.humbert:09} and conclude that
  $\Phi$ is smooth on $(M,g)$ and $D\Phi=0$ on $M$.  Since $\ker
  D_0=0$, we get that $\Phi \equiv 0$ and in particular $\Phi(p)=0$.
  This implies Equation (\ref{diff_mass}).
\end{proof}

\section{From existence to genericity}
\label{section_exis_to_gen}

The goal of this section is to prove the following Theorem.

\begin{theorem} \label{exis_to_gen} Let $M$ be a compact spin manifold
  of dimension $n$, $n \geq 3$, let $p \in M$ and let $U$ be a
  neighborhood of $p$. If $\cR_{p,U,\gfl}^{\neq 0} (M)$ is non-empty
  then it is generic in $\cR_{U,\gfl}(M)$.
\end{theorem}

\subsection{Continuity of the mass endomorphism}

The goal of this subsection is to prove that the mass endomorphism
depends continuously on $g$ in the $C^1$-topology.

\begin{proposition}
  Equip $\cR_{U,\gfl}^{\rm inv}(M)$ with the $C^1$-norm.  Then the map
  \begin{equation*}
    \cR_{U,\gfl}^{\rm inv} (M)\ni g
    \mapsto 
    \al^g \in \End(\Si_pM)
  \end{equation*}
  is continuous.
\end{proposition}

It follows that $\cR_{p,U,\gfl}^{\neq 0} (M)$ is open in
$\cR_{U,\gfl}^{\rm inv}(M)$ and thus in $\cR_{U,\gfl}(M)$.

\begin{proof}
  Let $(g_k)_{k\in\mN}$ be a family of metrics in $\cR_{U,\gfl}^{\rm inv}(M)$ 
  such that $g_k \to g$ in the $C^1$-topology. 
  For each $k$ the operator
  \begin{equation*}
    D^{g_k}_g 
    = 
    (\beta^g_{g_k})^{-1} D^{g_k} \beta^g_{g_k}
  \end{equation*}
  is invertible. We define
  \begin{equation*}
    P_k:=D^{g_k}_g - D^{g} .
  \end{equation*}
  Further, let $G^{g_k}$ and $G^g$ be the Green's functions of
  $D^{g_k}$ and $D^g$. We define
  \begin{equation*}
    Q_k := (\beta^g_{g_k})^{-1} G^{g_k} \beta^g_{g_k} - G^{g} .
  \end{equation*}
  Let $\psi\in\Si_pM$. 
  Using the equation (\ref{formel.green}) for $G^{g_k}$ and for $G^g$ 
  and using the fact that $g_k|_{U}=g|_U=\gfl$ we find that 
  \begin{equation*}
    Q_k\psi=(\be^g_{g_k})^{-1}v^{g_k}\be^g_{g_k}\psi-v^g\psi .
  \end{equation*}
  Therefore $Q_k\psi$ has a smooth continuation to all of $M$. 
  The equation $D^{g_k} G^{g_k} = D^gG^g = \delta_p \Id_{\Sigma_p M}$ then 
  tells us that
  \begin{equation*}
    Q_k = - (D^{g_k}_g)^{-1} P_k G^g.
  \end{equation*}
  $(G^g\psi)(x)$ becomes singular as $x\to p$. However we may take a smooth 
  function $\eta$ which is equal to $1$ near $p$ and has support in $U$ 
  and since $g_k|_{U}=g|_U=\gfl$ we obtain
  \begin{equation*}
    P_k(\eta G^g\psi)=D^g(\eta G^g\psi)-D^g(\eta G^g\psi)=0.
  \end{equation*}
  It follows that $P_kG^g\psi=P_k(1-\eta)G^g\psi$, where $(1-\eta)G^g\psi$ 
  is smooth on all of $M$.
  From (\ref{D^h_g_estimate}) it follows that the sequence 
  $(D^{g_k}_g)_{k\in\mN}$ converges to $D^g$ with respect to the norm of 
  bounded linear operators from $C^1(\Sigma^gM)$ to $C^0(\Sigma^gM)$.
  Therefore $\|P_kG^g\psi\|_{C^0}\to0$ as $k\to\infty$.
  Then it follows from \cite[Thm. IV-1.16]{kato:66} that 
  $((D^{g_k}_g)^{-1})_{k\in\mN}$ converges to $(D^{g})^{-1}$ 
  with respect to the norm of bounded linear operators from
  $C^0(\Sigma^gM)$ to $C^1(\Sigma^gM)$.
  Therefore $\|Q_k\psi\|_{C^1}\to0$ as $k\to\infty$.
  Evaluating $Q_k$ at $p$ yields $\al^{g_k}-\al^{g}$. 
  Thus the statement of the Proposition follows.
\end{proof}

\subsection{Real-analytic families of metrics}

Let $\varepsilon>0$. 
We say that a family $(g_t)_{t\in(-\varepsilon,\varepsilon)}$ 
of Riemannian metrics is real analytic if there exist 
sections $h_k$ of the bundle of symmetric bilinear forms 
on $M$, $k\in\mN$, such that for all 
$t\in(-\varepsilon,\varepsilon)$ and for all $r\in\mN$ we have 
$\|g_t-\sum_{k=0}^Nt^kh_k\|_{C^r}\to0$ as $N\to\infty$. 
Let $r$, $s\in\mN$.
A family $(P_t)_{t\in(-\varepsilon,\varepsilon)}$ of bounded 
linear operators $C^r(\Si^g M)\to C^s(\Si^g M)$ is called 
real analytic if there exist bounded linear operators 
$D_k$: $C^r(\Si^g M)\to C^s(\Si^g M)$, $k\in\mN$, such that 
for all $t$ we have 
$\|P_t-\sum_{k=0}^Nt^kD_k\|\to0$ as $N\to\infty$, where 
$\|.\|$ denotes the norm of bounded linear operators 
$C^r(\Si^g M)\to C^s(\Si^g M)$.

\begin{lemma}
  Let $M$ be closed and let $(g_t)_{t\in(-\varepsilon,\varepsilon)}$ 
  be a real analytic family of Riemannian metrics in 
  $\cR_{U,\gfl}^{\rm inv}(M)$. Then there exists 
  $\delta\in(0,\varepsilon]$ such that
  $(D^{g_t}_g)_{t\in(-\delta,\delta)}$ is a
  real analytic family of bounded linear 
  operators $C^1(\Si^gM)\to C^0(\Si^gM)$.
\end{lemma}

\begin{proof}
  Let $\|g_t-\sum_{k=0}^Nt^kh_k\|_{C^r}\to0$ as $N\to\infty$ for 
  all $r\in\mN$ and for all $t\in(-\varepsilon,\varepsilon)$.
  As in section \ref{identification_section} we define 
  endomorphisms $a^g_{g_t}$ and $a^g_{h_k}$, $k\in\mN$, of $TM$ 
  such that for all $X$, $Y$ in $TM$ we have
  \begin{equation*}
    g(a^g_{g_t}X,Y)=g_t(X,Y), \quad
    g(a^g_{h_k}X,Y)=h_k(X,Y).
  \end{equation*}
  Note that $a^g_{h_k}$ also exists if $h_k$ is not positive definite. 
  Let $|.|$ be the norm on $\Si^g M$ induced by the inner product 
  and let $\{e_i\}_{i=1}^n$ be a local $g$-orthonormal frame. 
  Since $(g_t)_{t\in(-\varepsilon,\varepsilon)}$ is real analytic 
  it follows that
  \begin{eqnarray*}
    &&\sup_{X\in TM,\,|X|=1} | a^g_{g_t}X - \sum_{k=0}^N t^k a^g_{h_k} X |\\
    &=&\sup_{X\in TM,\,|X|=1}
    | \sum_{i=1}^n g(a^g_{g_t}X,e_i) e_i 
    - \sum_{i=1}^n \sum_{k=0}^N t^k g(a^g_{h_k}X,e_i) e_i |\to0,
    \quad N\to\infty
  \end{eqnarray*}
  for all $t\in(-\varepsilon,\varepsilon)$. 
  In local coordinates one finds that for each $x\in M$ 
  there exists $\delta(x)\in(0,\varepsilon]$, 
  such that for all $X\in T_{x}M$ with $|X|=1$ the vector 
  $(a^g_{g_t})^{1/2}X$ is given by a power series which converges for 
  all $t\in(-\delta(x),\delta(x))$.
  Since $M$ is compact there exists $\delta\in(0,\varepsilon]$ such that 
  the convergence holds for all $X\in TM$, $|X|=1$ 
  and all $t\in(-\delta,\delta)$.
  Then after possibly decreasing $\delta$ a little further also 
  $b^g_{g_t}X$ for $t\in(-\delta,\delta)$ 
  is given by a power series which converges 
  uniformly in $X\in TM$, $|X|=1$. Furthermore for any vector fields 
  $X$, $Y$ the vector field $\nabla^{g_t}_{X}Y$ is also given 
  by a convergent power series as can be seen in local coordinates. 
  The assertion now follows from the 
  formula (\ref{D^h_g}) for $D^{g_t}_g$.
\end{proof}

\begin{proposition}
  If $(g_t)_{t\in (-\varepsilon,\varepsilon)}$ is a real-analytic
  family of metrics in $\cR_{U,\gfl}^{\rm inv}(M)$, then $\al^{g_t}$
  is also real-analytic.
\end{proposition}

\begin{proof} 
  It is sufficient to show that the family of operators
  $(\beta^g_{g_t})^{-1} G^{g_t} \beta^g_{g_t}$ is real analytic. 
  There exists $\delta\in(0,\varepsilon]$ such that 
  the family of operators 
  $(D^{g_t}_g)_{t\in(-\delta,\delta)}$ is real analytic. 
  It follows from \cite[VII-\S 1.1]{kato:66}
  that the family of operators 
  $((D^{g_t}_g)^{-1})_{t\in(-\delta,\delta)}$ 
  is also real analytic, possibly for some smaller $\delta$. 
  As above we define
  \begin{equation*}
    P_t := D^{g_t}_g - D^g , 
    \quad 
    Q_t := (\beta^g_{g_t})^{-1} G^{g_t} \beta^g_{g_t} - G^g 
  \end{equation*}
  and we obtain
  \begin{equation*}
    Q_t = -(D^{g_t}_g)^{-1} P_t G^{g}.
  \end{equation*}
  This completes the proof since the right hand side
  is real analytic.
\end{proof}

Consider a real analytic family $(g_t)_{t\in(a,b)}$ of Riemannian 
metrics on $M$. 
By unique continuation we immediately see: If there is a $t_0 \in
(a,b)$ with $\al^{g_{t_0}}\neq 0$, then the set
\begin{equation*}
  S \definedas
  \{t\in (a,b)\,|\,\al^{g_t}= 0\}
\end{equation*}
is a discrete subset of $(a,b)$.

Two metrics in the same connected component of $\cR_{U,\gfl}^{\rm
  inv}(M)$ can be joined by a piecewise real-analytic path of metrics.
It follows that if a connected component of $\cR_{U,\gfl}^{\rm inv}$
contains at least one metric with non-zero mass endomorphism, then the
metrics with non-zero mass endomorphism are dense in this component.
In order to obtain Theorem \ref{exis_to_gen}, we still have to discuss
families $(g_t)_{t \in (a,b)}$ where $D^{g_t}$ is not invertible for
some $t$. As the mass endomorphism is not defined for these $t$, we
complexify the parameter $t$ and pass around the metric with non
invertible $D^{g_t}$ in the imaginary direction. This is discussed in
the following subsection.

\subsection{Analytic continuation in the imaginary direction}

Again let $(g_t)_{t\in (a,b)}$ be a real-analytic family of metrics.
We assume $g_t\in\cR_{U,\gfl}(M)$ for any $t\in(a,b)$, but we do not
assume that all $D^{g_t}$ are invertible. Because of the
real-analyticity of $D^{g_t}_g$, the family can be extended to a
complex-analytic family of operators defined for $t$ in an open subset
$U\supset(a,b)$ of $\mC$.  In this complexification the operators
$D^{g_t}_g$ will no longer be self-adjoint, instead we have
$(D^{g_t}_g)^*=D^{g_{\overline t}}_g$.

As the set of invertible operators is open, we can assume without loss
of generality that $D^{g_t}_g$ is invertible on $U\setminus (a,b)$. In
other words we assume that
\begin{equation*}
  T \definedas
  \{t\in U\,|\, D^{g_t}_g\mbox{ is not invertible}\}
\end{equation*}
is contained in $(a,b)$.

The arguments from above also yield that $t\mapsto \al^{g_t}$ is a
holomorphic function on $U\setminus T$. As $U\setminus T$ is
connected, unique continuation implies the following Proposition.

\begin{proposition}
  If the mass endomorphism $\al^{g_{t_0}}$ is non-zero for any $t_0\in
  (a,b)\setminus T$, then
  \begin{equation*}
    \{t\in (a,b)\setminus T\,|\,\al^{g_t}\neq 0\}
  \end{equation*}
  is dense in $(a,b)$.
\end{proposition}

We will show in \cite{ammann.dahl.hermann.humbert:p10b} that the mass
endomorphism is actually meromorphic on $U$. The order of the poles in
$T$ is essentially the highest vanishing order of the eigenvalues
passing zero. These considerations also yield an alternative proof of
Theorem~\ref{mass_to_infty}, and thus indirectly the other
statements of the article.

\providecommand{\bysame}{\leavevmode\hbox to3em{\hrulefill}\thinspace}
\providecommand{\MR}{\relax\ifhmode\unskip\space\fi MR }
\providecommand{\MRhref}[2]{%
  \href{http://www.ams.org/mathscinet-getitem?mr=#1}{#2}
}
\providecommand{\href}[2]{#2}


\end{document}